\documentclass[a4paper,draft,reqno,12pt]{amsart}
\usepackage[utf8]{inputenc}
\usepackage[T1, T2A]{fontenc}
\usepackage[english]{babel}
\usepackage{amsmath}
\usepackage{amssymb}
\usepackage{amscd}
\usepackage{amsthm}
\usepackage{euscript}
\usepackage[all]{xy}
\newtheorem{proposition}{Proposition}

\newtheorem{lemma}{Lemma}
\newtheorem{theorem}{Theorem}

\newtheorem{corollary}{Corollary}
\theoremstyle{definition}

\theoremstyle{remark}
\newtheorem {remark}{Remark}

\DeclareMathOperator{\Aut}{Aut}

\DeclareMathOperator{\Der}{Der}

\def\GG{{\mathbb G}}

\def\KK{{\mathbb K}}

\def\ZZ{{\mathbb Z}}

\def\AA{{\mathbb A}}

\renewcommand{\phi}{\varphi}
\renewcommand{\ge}{\geqslant}

\begin{document}
\date{}
\title[Finiteness results on triangular automorphisms]{Some finiteness results on triangular automorphisms}
\author{Ivan Arzhantsev}
\thanks{The research was supported by Russian Science Foundation, grant 19-11-00172}
\address{National Research University Higher School of Economics, Faculty of Computer Science, Pokrovsky Boulevard 11, Moscow, 109028 Russia}
\email{arjantsev@hse.ru}
\author{Kirill Shakhmatov}
\address{National Research University Higher School of Economics, Faculty of Computer Science, Pokrovsky Boulevard 11, Moscow, 109028 Russia}
\email{bagoga@list.ru}

\maketitle
\begin{abstract}
In this note we prove that every finite collection of connected algebraic subgroups of the group of triangular automorphisms of the affine space generates a connected solvable algebraic subgroup.
\end{abstract}

%%%%%%%%%%%%%%%%%%%

\section{Introduction}
\label{sec1}

It is well known that the group $\Aut(\AA^n)$ of polynomial automorphisms of the affine space $\AA^n$ with $n\ge 2$ possesses a structure of an infinite-dimensional algebraic group; see \cite{Sh1,Sh2,FK}. Infinite-dimensionality of the automorphism group has many important consequences both for algebra and geometry. In particular, there are many results on the structure of subgroups of the group 
$\Aut(\AA^n)$, see, e.g. \cite{BNS,BEE,Po-1} and the references therein. Some of these results are similar with the case of (finite-dimensional) algebraic groups, but some of them are quite specific. It is natural to study finite-dimensional subgroups of the group  $\Aut(\AA^n)$. Let us say that a subgroup $G$ in $\Aut(\AA^n)$ is \emph{algebraic} if $G$ has a structure of an algebraic group such that the action $G\times\AA^n\to \AA^n$ is a morphism of algebraic varieties. In recent decades, a number of papers on this subject have naturally raised the following question: when does a finite collection of algebraic subgroups in the automorphism group generate an algebraic subgroup? In this note we address this question under the assumption that the subgroups are triangular. 

Let us recall that the group of polynomial automorphisms of the affine space $\AA^n$ can be identified with the group of automorphims of the polynomial algebra $\KK[x_1,\ldots,x_n]$. An~automorphsim 
$\varphi$ of the algebra $\KK[x_1,\ldots,x_n]$ is uniquely defined by the images $f_i:=\varphi(x_i)$ of the generators. Let us represent an automorphism $\varphi$ by a tuple $(f_1,\ldots,f_n)$ of polynomials. 

A tuple $(f_1,\ldots,f_n)$ of polynomials in $\KK[x_1,\ldots,x_n]$ is \emph{triangular} if we have $f_i=\lambda_ix_i+h_i$, where $\lambda_i$ are nonzero scalars, $h_i\in\KK[x_1,\ldots,x_{i-1}]$ for $i=2,\ldots,n$ and $h_1\in\KK$. An~automorphism $\varphi$ is \emph{triangular} if the corresponding tuple $(f_1,\ldots,f_n)$ is triangular. Clearly, any triangular tuple $(f_1,\ldots,f_n)$ defines an automorphism of the affine space. We say that a~subgroup $G$ of the automorphism group $\Aut(\AA^n)$ is \emph{triangular}, if $G$ consists of triangular automorphisms. 

\smallskip

Let us assume that the ground field $\KK$ is algebraically closed. The main result of this note is the following theorem. 

\begin{theorem} \label{theo}
Every collection $G_1,\ldots,G_s$ of connected triangular algebraic subgroups of the group $\Aut(\AA^n)$ generates a connected solvable algebraic subgroup of $\Aut(\AA^n)$ .
\end{theorem}

This note originates from an attempt to prove \cite[Proposition~3.6]{AZ}. Let $\GG_a$ be the additive group of the ground field or, equivalently, a one-dimensional unipotent algebraic group. 
In order to show that a finite collection of triangular $\GG_a$-subgroups generates a unipotent algebraic subgroup we used in~\cite{AZ} specific arguments including the Baker-Campbell-Hausdorff formula and the multivariate Zassenhaus formula. There is a desire to prove this fact by more direct methods.

\smallskip

In~\cite[Problem~3.1]{Po} it is asked when is the minimal closed subgroup of the group of polynomial automorphisms of the affine space containing two given $\GG_a$-subgroups finite-dimensional. Theorem~\ref{theo} shows that this is the case when the $\GG_a$-subgroups are triangular. 

\begin{corollary}
\label{cor1}
Every collection $U_1,\ldots,U_s$ of triangular $\GG_a$-subgroups generates a unipotent algebraic subgroup of $\Aut(\AA^n)$.
\end{corollary} 
 
One may be interested in a weaker property. Namely, let us say that an element $g$ of the automorphism group $\Aut(X)$ of an affine variety $X$ is \emph{algebraic} if $g$ is contained in an algebraic subgroup of the group $\Aut(X)$. In~\cite{PR}, the authors address the question when every element in a subgroup generated by a family of algebraic subgroups is algebraic. Some results in this direction are obtained in~\cite[Section~3]{PR}. It will be fascinating to study this question for subgroups generated by two $\GG_a$-subgroups in $\Aut(X)$.  

\smallskip

Now we assume that the ground field $\KK$ is an algebraically closed field of characteristic zero.
Let us come to an infinitesimal version of the results discussed above. It is well known that every derivation $D$ of the algebra $\KK[x_1,\ldots,x_n]$ has the form
$$
g_1\frac{\partial}{\partial x_1}+\ldots+g_n\frac{\partial}{\partial x_n}
$$
with some $g_1,\ldots,g_n\in\KK[x_1,\ldots,x_n]$. The vector space $\Der(A)$ of all derivations of an algebra $A$ forms a Lie algebra with respect to the commutator $[D_1,D_2]=D_1D_2-D_2D_1$. It is natural to ask when a set of derivations generates a finite-dimensional Lie algebra and which finite-dimensional Lie algebras appear this way. 

A derivation $D$ of an algebra $A$ is \emph{locally nilpotent} if  for every $a\in A$ is there a positive integer $k$ such that $D^k(a)=0$. It is an important problem to characterize collections of locally nilpotent derivations of $A$ that generate finite-dimensional Lie subalgebras in $\Der(A)$ and to describe finite-dimensional Lie algebras which can be realized this way. There are some results on this problem. For example, in~\cite[Section~5]{ALS} the case of homogeneous locally nilpotent derivations annihilating all fractions of homogeneous elements of the same degree in a graded algebra is studied and a description of finite-dimensional Lie algebras generated by such locally nilpotent derivations is obtained.

\smallskip

We say that a derivation $D$ of the algebra $\KK[x_1,\ldots,x_n]$ is \emph{triangular} if $g_i\in\KK[x_1,\ldots,x_{i-1}]$ for all $i=2,\ldots,n$ and $g_1\in\KK$. Clearly, every triangular derivation $D$ is locally nilpotent. See~\cite{Fr} for more information on locally nilpotent and triangular derivations.

\begin{corollary}
\label{cor2}
Every collection $D_1,\ldots,D_s$ of triangular derivations generates a finite-dimensional nilpotent Lie algebra.
\end{corollary} 

Concerning the proof of Theorem~\ref{theo}, one may observe that the degree of a composition of triangular automorphisms can be strictly higher than the degrees of the factors. For example, the square of the automorphism $(x_1,x_2+x_1^2,x_3+x_2^2)$ of degree 2 is 
$$
(x_1,x_2+2x_1^2,x_3+2x_2^2+2x_1^2x_2+x_1^4),
$$
so it has degree 4.  In Lemma~\ref{keylemma} we show that the degree of a product of triangular automorphisms of degree at most $m$ does not exceed $m^{n-1}$. This result essentially implies 
Theorem~\ref{theo} modulo technical details which are given in the next section. We also show that Theorem~\ref{theo} does not hold for non-connected algebraic subgroups.

\medskip

\emph{Acknowledgments.}\ We thank the referee for useful comments and corrections. Also the first author is grateful to Mikhail Zaidenberg for fruitful discussions and suggestions. 

%%%%%%%%%%%%%%%%%%%

\section{Proofs of the results}
\label{sec2}

We begin with a technical lemma. Let $\KK$ be an arbitrary field. By the degree of a mononial $x_1^{i_1}\ldots x_n^{i_n}$ we mean the sum $i_1+\ldots+i_n$. The degree of a polynomial in $\KK[x_1,\ldots,x_n]$ is the maximal degree of its terms. We denote by $\mathcal{T}(m)$ the set of triangular automorpshisms $(f_1,\ldots,f_n)$ such that the degree of any polynomial $f_i$ is at most $m$. Considering the coefficients of the polynomials $f_1,\ldots,f_n$ as coordinates, we obtain a structure of an affine algebraic variety on the set $\mathcal{T}(m)$ such that the action map $\mathcal{T}(m)\times \AA^n\to\AA^n$ is a morphism. 

\smallskip

Denote by $G(m)$ the subgroup of $\Aut(\AA^n)$ generated by $\mathcal{T}(m)$. Note that the inverse to an element from $\mathcal{T}(m)$ is a product of automorhisms of the form 
$(x_1,\ldots,\lambda_i^{-1}(x_i-h_i),\ldots,x_n)$, where $h_i$ is a monomial in $x_1,\ldots,x_{i-1}$ of degree at most $m$. It shows that any element of $G(m)$ is a product of elements of $\mathcal{T}(m)$. 

\begin{lemma}
\label{keylemma}
The subgroup $G(m)$ is contained in $\mathcal{T}(m^{n-1})$.
\end{lemma}

\begin{proof}
Let $x_{(0)}=\emptyset$ and $x_{(i)}$ be the tuple $(x_1,\ldots,x_i)$. We consider $h=(h_1,\ldots,h_n)\in G(m)$ with $h_i=\mu_ix_i+p_i$, where $\mu_i$ are nonzero scalars, $p_i\in\KK[x_1,\ldots,x_{i-1}]$ for $i=2,\ldots,n$ and $p_1\in\KK$. We define $hx_{(i)}$ as the tuple $(h_1,\ldots,h_i)$. Taking another tuple $h'=(h_1',\ldots,h_n')$ of the same form, we have
$$
h'hx_{(i)}=(\mu_j'\mu_jx_j+\mu_j'p_j(x_{(j-1)})+p_j'(hx_{(j-1)}))_{j=1}^i.
$$

Let us take any automorphism $(f_1,\ldots,f_n)$ from the group $G(m)$. We are going to prove that the degree of $f_i$ is at most $m^{i-1}$ for all $i=1,\ldots,n$. We proceed by induction on $i$. For $i=1$ we have
$f_1=\lambda x_1+c$, where $\lambda\in\KK\setminus\{0\}$ and $c\in\KK$, so the assertion holds. 

Assume it holds for $i-1$. We prove the assertion for $i$. It suffices to check that if it holds for $h(f_1,\ldots,f_n)$ for all $h=(\mu_jx_j+p_j(x_{(j-1)}))_{j=1}^n\in\mathcal{T}(m)$. 
Let $f_i=\lambda_ix_i+g_i$, where $\lambda_i\in\KK\setminus\{0\}$ and $g_i\in\KK[x_1,\ldots,x_{i-1}]$. We have
$$
(hf)_i=\mu_i\lambda_ix_i+\mu_ig_i(x_{(i-1)})+p_i(gx_{(i-1)}).
$$
By the inductive hypothesis, the degree of $p_i(gx_{(i-1)})$ does not exceed $m\cdot m^{i-2}=m^{i-1}$. The case $i=n$ completes the proof of the lemma. 
\end{proof}

\begin{remark}
While this paper was under review, we have found a variant of Lemma~\ref{keylemma} in \cite[Proposition~15.2.5]{FK}. 
\end{remark} 

From now on we assume that the ground field $\KK$  is algebraically closed. The following proposition is a version of \cite[Proposition~7.5]{Hum}. For convenience of the reader we provide it with a complete proof. 

\begin{proposition}
\label{1-prop}
For any positive integer $m$ the group $G(m)$ is a connected algebraic subgroup of $\Aut(\AA^n)$. Moreover, there is a positive integer $s$ such that every element of $G(m)$ has the form
$g_1\cdot\ldots\cdot g_s$ with some $g_i\in\mathcal{T}(m)$.
\end{proposition}

\begin{proof}
By Lemma~\ref{keylemma}, the subgroup $G(m)$ is contained in the affine variety $X:=\mathcal{T}(m^{n-1})$. We denote by $Y_l$ the image of the morphism 
$$
\mathcal{T}(m)^l\to X,\quad (g_1,\ldots,g_l)\mapsto g_1\cdot\ldots\cdot g_l.
$$
Consider the closure $\overline{Y_l}$ of $Y_l$ in $X$. This is an irreducible closed subvariety in $X$ and $\overline{Y_l}\subseteq\overline{Y_{l+1}}$ for any $l$. Then there is a positive integer $a$ such that
$\overline{Y_a}=\overline{Y_{a+1}}=\ldots$. 

We claim that $\overline{Y_b}\cdot\overline{Y_c}$ is contained in $\overline{Y_{b+c}}$ for all positive integers $b$ and $c$. Indeed, let $x_0\in Y_c$ and consider the map $Y_b\to Y_{b+c}$ given by $y\mapsto y\cdot x_0$. By Lemma~\ref{keylemma}, this map extends to a morphism $\mathcal{T}(m^{n-1})\to \mathcal{T}(m^{(n-1)^2})$. Restricting it to $\overline{Y_b}$, we obtain the morphism
$\overline{Y_b}\to\overline{Y_{b+c}}$. So we have $\overline{Y_b}\cdot Y_c\subseteq\overline{Y_{b+c}}$. 

Let $y_0\in\overline{Y_b}$ and consider the map $Y_c\to\overline{Y_{b+c}}$ given by $x\mapsto y_0\cdot x$. By the arguments given above this map extends to the morphism $\overline{Y_c}\to\overline{Y_{b+c}}$. So we obtain $\overline{Y_b}\cdot \overline{Y_c}\subseteq \overline{Y_{b+c}}$. 

In particular, the subset $\overline{Y_a}$ contains the subgroup $G(m)$ and $\overline{Y_a}\cdot\overline{Y_a}\subseteq \overline{Y_{2a}}=\overline{Y_a}$. Since $Y_a$ is a constructible subset, it contains a dense open subset $U\subseteq\overline{Y_a}$. We claim that the subset $U\cdot U$ coincides with $\overline{Y_a}$. Indeed, the image $U^{-1}$ of the subset $U$ under the automorphism of taking the inverse element is also an open subset of $\overline{Y_a}$, and for any element $g\in\overline{Y_a}$ the intersection of subsets $U$ and $g\cdot U^{-1}$ is nonempty. It shows that $g=h_1\cdot h_2$ with some $h_1,h_2\in U$. 

We conclude that the subgroup $G(m)$ coincides with $\overline{Y_a}=Y_a\cdot Y_a$ and the second assertion of the proposition holds with $s=2a$.
\end{proof} 

\begin{proposition}
\label{2-prop}
Let $m$ be a positive integer and $f_i\colon X_i\to\mathcal{T}(m)$ with $i\in\mathcal{I}$ be morphisms from irreducible algebraic varieties $X_i$. Assume that every image $Y_i:=f_i(X_i)$  contains the unit element of $\mathcal{T}(m)$. Then the subsets $\{Y_i, i\in\mathcal{I}\}$ generate in $\Aut(\AA^n)$ a connected algebraic subgroup $G_{\mathcal{I}}$. Moreover, there is a finite sequence $I=(i_1,\ldots,i_k)$ of indices in $\mathcal{I}$ such that $G_{\mathcal{I}}=Y_{i_1}^{\epsilon_1}\cdot\ldots\cdot Y_{i_k}^{\epsilon_k}$ with $\epsilon_j=\pm 1$. 
\end{proposition}

\begin{proof}
Since $\mathcal{T}(m)$ is contained in the algebraic subgroup $G(m)$ as a closed subset, we can assume that the morphisms $f_i$ map $X_i$ to $G(m)$. Now the asserts ion follows from \cite[Proposition~7.5]{Hum}.
\end{proof} 

\begin{proof}[Proof of Theorem~\ref{theo}]
We take a subgroup $G_j$ from the collection $G_1,\ldots,G_s$. By \cite[Lemma~1.4]{PV}, any variable $x_i$ is contained in a finite-dimensional subspace in $\KK[x_1,\ldots,x_n]$ that is invariant under the action of $G_j$. This implies that $G_j$ is contained in $\mathcal{T}(m)$ for some positive integer $m$. The orbit map that applies an element of $G_j$ to the tuple $(x_1,\ldots,x_n)$ defines a morphism $f_j\colon G_j\to\mathcal{T}(m)$. Let us take the maximal value of $m$ over all subgroups $G_1,\ldots,G_s$. Then it follows from Proposition~\ref{2-prop} that the subgroups $G_1,\ldots,G_s$ generate a connected algebraic subgroup.

Finally, we observe that the group of all triangular automorphisms is solvable, see e.g. \cite[Theorem~2, 2)]{BNS}~\footnote{The results of~\cite{BNS} are obtained under the assumption that the ground field has characteristic zero. But it is clear that over any field the first derived subgroup of the group of triangular automorphisms is contained in the group of unitriangular automorphisms, while the $(s+1)$th derived 
subgroup fixes the variables $x_1,\ldots, x_s$; see \cite{BNS} for details. In particular, the $(n+1)$th derived subgroup is trivial.}. This implies that any group of triangular automorphisms is solvable as well.
\end{proof}

\begin{proof}[Proof of Corollary~\ref{cor1}]
By Theorem~\ref{theo}, the subgroups $U_1,\ldots,U_s$ generate an algebraic group~$G$. For any eigenvector $F\in\KK[x_1,\ldots,x_n]$ of an operator $g\in G$ we can order monomials in
such a way that $g$ sends any term in $F$ to this term plus a linear combination of earlier monomials. It shows  that the eigenvalue of $F$ equals $1$. So all elements of $G$ are unipotent and,
by definition, $G$ is a unipotent group. 
\end{proof}

\begin{proof}[Proof of Corollary~\ref{cor2}]
 Let $U_i$ be the triangular $\GG_a$-subgroup $\{\exp(sD_i); s\in\KK\}$.  By Corollary~\ref{cor1}, the subgroups $U_1,\ldots,U_s$ generate a unipotent algebraic group $G$. The derivations $D_1,\ldots,D_s$ lie in the tangent algebra of the group $G$, so they generate a finite-dimensional nilpotent Lie algebra. 
\end{proof}

\begin{remark}
Assume the ground field $\KK$ has characteristic zero. We show that Theorem~\ref{theo} does not hold if the subgroups $G_1,\ldots,G_s$ are non-connected. Indeed, consider two triangular matrices
$$A=
\begin{pmatrix}
1 & a \\
0 & -1
\end{pmatrix} 
\quad
\text{and}
\quad B=
\begin{pmatrix}
1 & b \\
0 & -1
\end{pmatrix} 
$$
with $a\ne b$. Each of these matrices generates a subgroup of order two. Let $G$ be the subgroup generated by $A$ and $B$. The set of matrices in $G$ with determinant $1$ is 
$$
 \left\{\begin{pmatrix}
1 & k(a-b) \\
0 & 1
\end{pmatrix}; k\in\ZZ\right\}.
$$ 
This is an infinite proper subgroup of the connected one-dimensional group 
$\left\{\begin{pmatrix}
1 & c \\
0 & 1
\end{pmatrix}; c\in\KK\right\}.$ This proves that $G$ is not an algebraic group. 
\end{remark}

%%%%%%%%%%%%%%%%%%%%%%%%%%%%%%%%%%%%%%%%%%%%%%

\end{document}